\documentclass[11pt]{amsart}

\usepackage{amsmath,amssymb}
\usepackage{color}

\theoremstyle{definition}

\theoremstyle{proposition}
\newtheorem{thm}{Theorem}
\newtheorem{lem}{Lemma}

\begin{document}

\title[Generalized Kronecker formula]{Generalized Kronecker formula for Bernoulli numbers and self-intersections of curves on a surface}

\author{Shinji Fukuhara}
\address{Department of Mathematics, Tsuda College,
2-1-1 Tsuda-machi, Kodaira-shi Tokyo 187-8577, Japan}
\email{fukuhara@tsuda.ac.jp}

\author{Nariya Kawazumi}
\address{Department of Mathematical Sciences, University of Tokyo,
3-8-1 Komaba, Meguro-ku Tokyo 153-8914, Japan}
\email{kawazumi@ms.u-tokyo.ac.jp}

\author{Yusuke Kuno}
\address{Department of Mathematics, Tsuda College,
2-1-1 Tsuda-machi, Kodaira-shi Tokyo 187-8577, Japan}
\email{kunotti@tsuda.ac.jp}

\subjclass[2010]{Primary 11B68, 57N05, 57M99;}
\keywords{Bernoulli numbers, surfaces, self-intersections}
\thanks{The authors would like to thank Professor Noriko Yui for useful comments to an earlier version of this paper.
S.\ F.\ is supported by JSPS KAKENHI (No.26400098).
N.\ K.\  is partially supported by JSPS KAKENHI (No.24224002), (No.24340010) and (No.15H03617).
Y.\ K.\  is supported by JSPS KAKENHI (No.26800044).} 

\date{\today}

\maketitle

\begin{abstract}
We present a new explicit formula for the $m$-th Bernoulli number $B_m$, which involves two integer parameters $a$ and $n$ with $0\le a\le m\le n$.
If we set $a=0$ and $n=m$, then the formula reduces to the celebrated Kronecker formula for $B_m$.
We give two proofs of our formula.
One is analytic and uses a certain function in two variables.
The other is algebraic and is motivated by a topological consideration of self-intersections of curves on an oriented surface.
\end{abstract}

\section{Introduction}

The Bernoulli numbers $B_m\ (m\ge 0)$ are defined by the generating function
\begin{equation*}
  \frac{x}{e^x-1}=\sum_{m=0}^{\infty} \frac{B_m}{m!} x^m.
\end{equation*}

We have: $B_0=1, B_1=-1/2, B_2=1/6, B_4=-1/30,\ldots$, and $B_m=0$ for all odd $m\ge 3$.
A large number of identities involving the Bernoulli numbers has been known \cite{DLS} \cite{GO1} \cite{NI1} \cite{SA1}.
Most of them give relationships between $B_m$ and $B_i\,(0\leq i<m)$. 
These identities provide various ways to compute $B_m$ recursively from the $B_i$'s for $0\le i<m$.

Contrary to the above recursive approach, the following formula of Kronecker gives a direct method for computing $B_m$.
\begin{thm}[Kronecker \cite{KR1}, see also \cite{GO1} \cite{HI1} \cite{NI1} \cite{SA1}]\label{thm1}
For any integer $m\ge 2$, it holds that
\begin{equation}
\label{kronecker}
  B_m=\sum_{k=1}^{m+1}\frac{(-1)^{k+1}}{k}\binom{m+1}{k}\sum_{i=1}^{k-1}i^m.
\end{equation}
\end{thm}

In this article we generalize the formula (\ref{kronecker}) to a formula with two parameters:
\begin{thm}
\label{thm2}
Let $m,n,a$ be integers satisfying $\ 0\leq a\leq m\leq n$.
Then it holds that
\begin{equation}
\label{general}
  B_m=(-1)^{a}\sum_{k=1}^{n+1}\frac{(-1)^{k+1}}{k}\binom{n+1}{k}\left[\sum_{i=1}^{k-1}i^a(k-i)^{m-a}+\delta_{a,m}k^m \right].
\end{equation}
Here $\delta_{a,m}$ is the Kronecker delta.
Furthermore, if $m\ge 2$, it holds that
\begin{equation}
\label{simple}
  B_m=(-1)^{a}\sum_{k=1}^{n+1}\frac{(-1)^{k+1}}{k}\binom{n+1}{k}
\sum_{i=1}^{k-1}i^a(k-i)^{m-a}.
\end{equation}
\end{thm}

It is clear that, in the case $a=0$ and $n=m$, the formula \eqref{simple} reduces to the Kronecker formula \eqref{kronecker}.

We give two proofs of Theorem \ref{thm2}.
In \S 2, we introduce a two-variable function $g(x,y)$ and compute its series expansion in two different ways.
This leads to a proof of Theorem \ref{thm2}.
In \S 3, we construct a certain continuous map $\hat{\mu}\colon \mathbb{Q}[[Z]]\to \mathbb{Q}[[X,Y]]$ between the rings of formal power series.
A key observation is that $\hat{\mu}(Z)$ is expressed in terms of the Bernoulli numbers, and this leads to another proof of Theorem \ref{thm2}.

The map $\hat\mu$ is motivated by an operation $\mu$ to a curve on an oriented surface.
This operation was introduced in \cite{KK1} inspired by a construction of Turaev \cite{T1}, and, among other things, it computes self-intersections of curves.
In \S4 we first recall the operation $\mu$ from \cite{KK1}. Then we obtain an exact formula for $\mu$ (Theorem \ref{thmulog}) based on the results in \S3.
The Bernoulli numbers have already appeared in the tensorial description of the homotopy intersection form  on an oriented surface \cite{MT}.
Our formula provides yet another evidence for a close connection between topology of surfaces and Bernoulli numbers.

\section{The first proof}

Let $f(x,y)$ and $g(x,y)$ be functions in variables $x$ and $y$ defined by
\begin{equation*}\label{f(x,y)}
  f(x,y):=\int_{x}^{y}(e^t-1)^{n+1}dt, \quad \text{and} \quad
  g(x,y):=\frac{f(x,y)}{e^{y-x}-1}.
\end{equation*}
We will examine the coefficient of $x^ay^{m-a}$ in the series expansion of $g(x,y)$.

First we compute $f(x,y)$ as follows:
\begin{align*}
   f(x,y)&=\int_{x}^{y}(e^t-1)^{n+1}dt \\
   &=\int_{x}^{y}\sum_{k=0}^{n+1}(-1)^{n+1-k}\binom{n+1}{k}e^{kt}dt \\
   &=(-1)^{n+1}\sum_{k=1}^{n+1}\frac{(-1)^k}{k}
     \binom{n+1}{k}(e^{ky}-e^{kx})+(-1)^{n+1}(y-x).\\
\end{align*}
Since
\begin{equation*}
  \frac{e^{ky}-e^{kx}}{e^{y-x}-1}
  =\frac{e^{kx}(e^{k(y-x)}-1)}{e^{y-x}-1}
  =\sum_{i=1}^{k-1}e^{ix}e^{(k-i)y}+e^{kx},
\end{equation*}
we can compute $g(x,y)$ as follows: 
\begin{align*}
  g(x,y)=&\frac{f(x,y)}{e^{y-x}-1} \\
   =&(-1)^{n+1}\sum_{k=1}^{n+1}\frac{(-1)^k}{k}
     \binom{n+1}{k}\frac{(e^{ky}-e^{kx})}{e^{y-x}-1} 
     +(-1)^{n+1}\frac{y-x}{e^{y-x}-1}  \\
   =&(-1)^{n+1}\sum_{k=1}^{n+1}\frac{(-1)^k}{k}
     \binom{n+1}{k}\left[\sum_{i=1}^{k-1}e^{ix}e^{(k-i)y}+e^{kx}\right]\\
     &+(-1)^{n+1}\sum_{b=0}^{\infty}\frac{B_b}{b!}(y-x)^b. \\
\end{align*}
Then using the identities:
$$e^{ix}e^{(k-i)y}=\sum_{b,c=0}^{\infty} \frac{i^b(k-i)^c}{b!c!}x^by^c
\quad \text{and} \quad e^{kx}=\sum_{b=0}^{\infty}\frac{k^b}{b!}x^b,$$
we see that the coefficient of $x^ay^{m-a}$ in $g(x,y)$ is given by
\begin{align*}
&(-1)^{n+1}\sum_{k=1}^{n+1}\frac{(-1)^k}{k}
     \binom{n+1}{k}
     \left[\sum_{i=1}^{k-1}\frac{i^a}{a!}\frac{(k-i)^{m-a}}{(m-a)!}
     +\delta_{a,m}\frac{k^m}{m!}\right] \\
     &+(-1)^{n+1+a}\frac{B_m}{m!}
        \binom{m}{a}.\\
\end{align*}
This is equal to $((-1)^{n+1+a}/m!)\binom{m}{a}$ times
\begin{equation}
\label{coeff1}
(-1)^{a}\displaystyle\sum_{k=1}^{n+1}\frac{(-1)^k}{k}
\binom{n+1}{k}\left[\sum_{i=1}^{k-1}i^a(k-i)^{m-a}
+\delta_{a,m}k^m\right]+B_m.
\end{equation}

Secondly, we expand $g(x,y)$ in a different way.
Put $g_1(x,y)=f(x,y)/(y-x)$.
Then we have
\begin{equation*}
  g(x,y)=\frac{f(x,y)}{y-x}\frac{y-x}{e^{y-x}-1}=g_1(x,y) \sum_{b=0}^{\infty}\frac{B_b}{b!}(y-x)^b.
\end{equation*}
Writing $(e^t-1)^{n+1}=\sum_{i\ge n+1} a_it^i$, we have
$$f(x,y)=\int_x^y(e^t-1)^{n+1}dt=\sum_{i\ge n+1}\frac{a_i}{i+1}(y^{i+1}-x^{i+1}).$$
Thus the series expansion of $g_1(x,y)$ has all terms of degree $\ge n+1$, so does that of $g(x,y)$.
In particular, the coefficient of $x^ay^{m-a}$ in this expansion is zero.
Therefore, the expression (\ref{coeff1}) is zero, and we obtain the formula (\ref{general}).

Finally, we can derive the formula (\ref{simple}) in Theorem \ref{thm2} from the formula (\ref{general}) by applying the following lemma. 
Although it might be well known, we give its proof for the sake of completeness.
\begin{lem}
\label{lem1}
Let $m,n$ be integers satisfying $\ 0\leq m\leq n$.
Then it holds that
\begin{equation*}
  \sum_{k=1}^{n+1}(-1)^{k}\binom{n+1}{k}k^m =
     \begin{cases}
     0 & \text{if \ $m\geq 1$,} \\
      -1 & \text{if \ $m=0$}.
    \end{cases} 
\end{equation*}
\end{lem}
\begin{proof}
Set $f(x):=(e^x-1)^{n+1}$.
Since $m\le n$, the coefficient of $x^m$ in the series expansion of $f(x)$ is zero.

On the other hand, we compute
\begin{align*}
  f(x)&=\sum_{k=0}^{n+1}(-1)^{n+1-k}\binom{n+1}{k}e^{kx} \\
  &=(-1)^{n+1}\left[\sum_{k=1}^{n+1}(-1)^{k}\binom{n+1}{k}e^{kx}
    +1\right]  \\
  &=(-1)^{n+1}\left[\sum_{k=1}^{n+1}(-1)^{k}\binom{n+1}{k}\sum_{a=0}^{\infty}
    \frac{k^a}{a!}x^a+1\right] . \\
\end{align*}
Since the coefficient of $x^m$ in the last expression is equal to
$$     \begin{cases}
     \displaystyle\frac{ (-1)^{n+1}}{m!}\sum_{k=1}^{n+1}(-1)^{k}\binom{n+1}{k}
       k^m & \text{if $m\geq 1$,} \\
      \displaystyle(-1)^{n+1}\left[\sum_{k=1}^{n+1}(-1)^{k}\binom{n+1}{k}+1\right]
       & \text{if \ $m=0$},
    \end{cases} 
$$
the assertion follows.
\end{proof}

This completes the proof of Theorem \ref{thm2}.

\section{The second proof}
First of all, we describe a preliminary construction.

Let $\mathbb{Q}[[Z]]$ (resp. $\mathbb{Q}[[X,Y]]$) be the ring of formal power series in an indeterminate $Z$ (resp. in indeterminates $X$ and $Y$).
For a non-negative integer $p$, let $F_p^Z$ (resp. $F_p^{X,Y}$) be the set of formal power series in $\mathbb{Q}[[Z]]$ (resp. $\mathbb{Q}[[X,Y]]$) which has only terms of (total) degree $\ge p$.
We have natural isomorphisms $\mathbb{Q}[[Z]] \cong \varprojlim_p \mathbb{Q}[[Z]]/F_p^Z$ and $\mathbb{Q}[[X,Y]] \cong \varprojlim_p \mathbb{Q}[[X,Y]]/F_p^{X,Y}$.

Set $z:=e^Z=\sum_{i=0}^{\infty}(1/i!)Z^i$.
Then the Laurent polynomial ring $\mathbb{Q}[z,z^{-1}]$ is a subring of $\mathbb{Q}[[Z]]$.
The augmentation ideal $I$ is defined by
$$I={\rm Ker}(\mathbb{Q}[z,z^{-1}] \to \mathbb{Q}, \sum_j a_jz^j \mapsto \sum_j a_j).$$
Then $I$ gives a filtration $\{ I^p \}_p$ of $\mathbb{Q}[z,z^{-1}]$.
By the inclusion map $\mathbb{Q}[z,z^{-1}]\hookrightarrow \mathbb{Q}[[Z]]$, the filtration $\{ F_p^Z \}_p$ restricts to $\{ I^p \}_p$.
Moreover, we have a natural isomorphism $\mathbb{Q}[[Z]] \cong \varprojlim_p \mathbb{Q}[z,z^{-1}]/I^p$.

Define a $\mathbb{Q}$-linear map $\hat{\mu}\colon \mathbb{Q}[z,z^{-1}] \to \mathbb{Q}[[X,Y]]$ by
\begin{equation}
\label{eq:dfn-mu}
\hat{\mu}(z^k)=\begin{cases}
-\sum_{i=1}^k e^{iX}e^{(k-i)Y} & (k>0) \\
0 & (k=0) \\
\sum_{i=0}^{|k|-1}e^{-iX}e^{(k+i)Y} & (k<0).
\end{cases}
\end{equation}
From the definition of $\hat{\mu}$ it is easy to see that
$$(e^{-X}e^Y-1)\hat{\mu}(z^k)=e^{kX}-e^{kY}, \quad k\in \mathbb{Z}.$$
Therefore, we have
\begin{equation}
\label{eq:fundamental}
(e^{-X}e^Y-1)\hat{\mu}(f(z))=f(e^X)-f(e^Y)
\end{equation}
for any Laurent polynomial $f(z)\in \mathbb{Q}[z,z^{-1}]$.
Consider
$$\Phi(X,Y):=\sum_{i=0}^{\infty} \frac{B_i}{i!}(-X+Y)^i.$$
Then we have $(e^{-X}e^Y-1)\Phi(X,Y)=-X+Y$.
Multiplying $\Phi(X,Y)$ to the both sides of (\ref{eq:fundamental}), we have
\begin{equation}
\label{eq:fundamental2}
(-X+Y)\hat{\mu}(f(z))=(f(e^X)-f(e^Y))\Phi(X,Y)
\end{equation}
for any $f(z)\in \mathbb{Q}[z,z^{-1}]$.

\begin{lem}
There is a unique continuous extension $\hat{\mu}\colon \mathbb{Q}[[Z]]\to \mathbb{Q}[[X,Y]]$ of the map $\hat{\mu}$ in {\rm (\ref{eq:dfn-mu})}.
\end{lem}

\begin{proof}
It is sufficient to prove that $\hat{\mu}(I^p)\subset F_{p-1}^{X,Y}$ for any $p\ge 1$.
Suppose $f(z)\in I^p$.
Then $f(e^X)$ and $f(e^Y)$ lie in $F_p^{X,Y}$.
This means that the right hand side of (\ref{eq:fundamental2}) is an element of $F_p^{X,Y}$.
Therefore, $\hat{\mu}(f(z))\in F_{p-1}^{X,Y}$.
\end{proof}

Now for each $k\ge 1$ we can put $f(z)=(\log z)^k=Z^k$ in (\ref{eq:fundamental2}), and we obtain
$$(-X+Y)\hat{\mu}(Z^k)=(X^k-Y^k)\Phi(X,Y).$$
This shows that $\hat{\mu}(Z^k)\in F_{k-1}^{X,Y}$.
Setting $k=1$, we have
\begin{equation}
\hat{\mu}(Z) =-\Phi(X,Y)
=-\sum_{i=0}^{\infty}\frac{B_i}{i!}
\sum_{j=0}^i (-1)^j \binom{i}{j} X^j Y^{i-j}.
\label{eq:mu(Z)}
\end{equation}

\begin{proof}[The second proof of Theorem \ref{thm2}]
We will give another proof to the formula (\ref{general}) alone.
In what follows, $\equiv$ means an equality in $\mathbb{Q}[[X,Y]]$ modulo $F_{n+1}^{X,Y}$.
For $k=1,\ldots,n+1$, we have
\begin{equation}
\label{eq:mu(z^k)}
\hat{\mu}(z^k)=\hat{\mu}(e^{kZ})=\sum_{i=1}^{\infty} \frac{k^i}{i!}\hat{\mu}(Z^i)
\equiv \sum_{i=1}^{n+1} \frac{k^i}{i!}\hat{\mu}(Z^i).
\end{equation}

Consider the square matrix $D=(D_{ki})_{k,i}$ of order $n+1$, where $D_{ki}=k^i/i!$.
Then $D$ is invertible, and the inverse matrix of $D$ has the first row $(a_1,\ldots,a_{n+1})$, where
$$a_k=\frac{(-1)^{k+1}}{k}\binom{n+1}{k}$$
(see also Lemma \ref{lem1}).
From (\ref{eq:mu(z^k)}) we have
\begin{equation}
\label{eq:Zz}
\hat{\mu}(Z)\equiv \sum_{k=1}^{n+1} a_k \hat{\mu}(z^k)= \sum_{k=1}^{n+1}\frac{(-1)^{k+1}}{k} \binom{n+1}{k} \hat{\mu}(z^k).
\end{equation}
Furthermore, for $k=1,\ldots,n+1$, from (\ref{eq:dfn-mu}) we have
\begin{equation}
\label{eq:Zz2}
\hat{\mu}(z^k)=-\sum_{i=1}^{k-1}\sum_{a,b=0}^{\infty}\frac{i^a(k-i)^b}{a!b!} X^a Y^b
-\sum_{a=0}^{\infty} \frac{k^a}{a!}X^a.
\end{equation}
By (\ref{eq:Zz}) and (\ref{eq:Zz2}), the coefficient of $X^aY^{m-a}$ in $\hat{\mu}(Z)$ is
$$\sum_{k=1}^{n+1}\frac{(-1)^k}{k}\binom{n+1}{k}
\left[ \sum_{i=1}^{k-1} \frac{i^a(k-i)^{m-a}}{a!(m-a)!}+\delta_{m,a}\frac{k^m}{m!} \right].$$
On the other hand, by (\ref{eq:mu(Z)}), this coincides with
$$(-1)^{a+1}\frac{B_m}{m!}\binom{m}{a}=\frac{(-1)^{a+1}}{a!(m-a)!}B_m.$$
This completes the proof. 
\end{proof}

\section{A topological background for the second proof}

Let $S$ be a compact connected oriented surface with $\partial S\neq \emptyset$.
Fix a basepoint $*\in \partial S$ and set $\pi_1(S):=\pi_1(S,*)$.
We denote by $\hat{\pi}(S)$ the set of free homotopy classes of oriented loops on $S$.
For any $p\in S$, we denote by $|\ |\colon \pi_1(S,p)\to \hat{\pi}(S)$ the forgetful map of the basepoint.

We recall the operation $\mu\colon \mathbb{Q}\pi_1(S) \to \mathbb{Q}\pi_1(S)\otimes (\mathbb{Q}\hat{\pi}(S)/\mathbb{Q}{\bf 1})$, which has been introduced in \cite{KK1} inspired by a construction of Turaev \cite{T1}.
Here, ${\bf 1}$ is the class of a constant loop.
Let $\gamma\colon [0,1]\to S$ be an immersed based loop.
We arrange so that the pair of tangent vectors $(\dot{\gamma}(0),\dot{\gamma}(1))$ is a positive basis of the tangent space $T_*S$, and that the self-intersections of $\gamma$ (except for the base point $*$) lie in the interior ${\rm Int}(S)$ and consist of transverse double points.
Let $\Gamma$ be the set of double points of $\gamma$.
For $p\in \Gamma$ we denote $\gamma^{-1}(p)=\{ t_1^p,t_2^p\}$, so that $0<t_1^p<t_2^p<1$.
We define
$$\mu(\gamma):=-\sum_{p\in \Gamma} \varepsilon(\dot{\gamma}(t_1^p),\dot{\gamma}(t_2^p))
(\gamma_{0t_1^p}\gamma_{t_2^p1})\otimes |\gamma_{t_1^p t_2^p}|
\in \mathbb{Q}\pi_1(S)\otimes (\mathbb{Q}\hat{\pi}(S)/\mathbb{Q}{\bf 1}).$$
Here,
\begin{itemize}
\item the sign $\varepsilon(\dot{\gamma}(t_1^p),\dot{\gamma}(t_2^p))$ is $+1$ if the pair $(\dot{\gamma}(t_1^p),\dot{\gamma}(t_2^p))$ is a positive basis of $T_pS$, and is $-1$ otherwise,
\item the based loop $\gamma_{0t_1^p}\gamma_{t_2^p1}$ is the conjunction of the paths $\gamma|_{[0,t_1^p]}$ and $\gamma|_{[t_2^p,1]}$,
\item the element $\gamma_{t_1^pt_2^p}\in \pi_1(S,p)$ is the restriction of $\gamma$ to $[t_1^p,t_2^p]$ and we understand that $|\gamma_{t_1^p t_2^p}|=0$ if the loop $\gamma_{t_1^p t_2^p}$ is homotopic to a constant loop.
\end{itemize}
We remark that the alternating part of $(|\ |\otimes 1)\mu(\gamma)$ is exactly the Turaev cobracket \cite{T2} of the free loop $|\gamma|$.

We observe that if $\gamma$ is simple under the condition that the pair $(\dot{\gamma}(0),\dot{\gamma}(1))$ is a positive basis of $T_*S$, then for any integer $k\in \mathbb{Z}$,
\begin{equation}
\label{eq:mu(gamma^k)}
\mu(\gamma^k)=
\begin{cases}
-\sum_{i=1}^{k-1} \gamma^i \otimes |\gamma^{k-i}| & (k>0) \\
0 & (k=0) \\
\sum_{i=0}^{|k|-1} \gamma^{-i} \otimes |\gamma^{k+i}| & (k<0). \\
\end{cases}
\end{equation}
See Figure 1.
The definition of $\hat{\mu}$ in (\ref{eq:dfn-mu}) is motivated by this formula.
\begin{figure}
\caption{computation of $\mu(\gamma^k)$ for simple $\gamma$ ($k=4$)}
\unitlength 0.1in
\begin{picture}( 16.8000, 16.3000)( 19.5000,-28.3000)
%
{\color[named]{Black}{%
\special{pn 8}%
\special{ar 2630 2640 160 160  1.5707963 3.1415927}%
}}%
%
{\color[named]{Black}{%
\special{pn 8}%
\special{pa 2470 2640}%
\special{pa 2470 1840}%
\special{fp}%
}}%
%
{\color[named]{Black}{%
\special{pn 8}%
\special{ar 2630 1840 160 160  3.1415927 4.7123890}%
}}%
%
{\color[named]{Black}{%
\special{pn 8}%
\special{pa 2630 1680}%
\special{pa 2950 1680}%
\special{fp}%
}}%
%
{\color[named]{Black}{%
\special{pn 8}%
\special{ar 2950 1840 160 160  4.7123890 6.2831853}%
}}%
%
{\color[named]{Black}{%
\special{pn 8}%
\special{ar 2950 2160 160 160  6.2831853 6.2831853}%
\special{ar 2950 2160 160 160  0.0000000 1.5707963}%
}}%
%
{\color[named]{Black}{%
\special{pn 8}%
\special{pa 3110 2160}%
\special{pa 3110 1840}%
\special{fp}%
}}%
%
{\color[named]{Black}{%
\special{pn 8}%
\special{pa 2950 2320}%
\special{pa 2470 2320}%
\special{fp}%
}}%
%
{\color[named]{Black}{%
\special{pn 8}%
\special{ar 2470 2160 160 160  1.5707963 3.1415927}%
}}%
%
{\color[named]{Black}{%
\special{pn 8}%
\special{pa 2310 2160}%
\special{pa 2310 1840}%
\special{fp}%
}}%
%
{\color[named]{Black}{%
\special{pn 8}%
\special{ar 2630 1840 320 320  3.1415927 4.7123890}%
}}%
%
{\color[named]{Black}{%
\special{pn 8}%
\special{pa 2630 1520}%
\special{pa 2950 1520}%
\special{fp}%
}}%
%
{\color[named]{Black}{%
\special{pn 8}%
\special{ar 2950 1840 320 320  4.7123890 6.2831853}%
}}%
%
{\color[named]{Black}{%
\special{pn 8}%
\special{pa 3270 1840}%
\special{pa 3270 2160}%
\special{fp}%
}}%
%
{\color[named]{Black}{%
\special{pn 8}%
\special{ar 2950 2160 320 320  6.2831853 6.2831853}%
\special{ar 2950 2160 320 320  0.0000000 1.5707963}%
}}%
%
{\color[named]{Black}{%
\special{pn 8}%
\special{pa 2950 2480}%
\special{pa 2470 2480}%
\special{fp}%
}}%
%
{\color[named]{Black}{%
\special{pn 8}%
\special{ar 2470 2160 320 320  1.5707963 3.1415927}%
}}%
%
{\color[named]{Black}{%
\special{pn 8}%
\special{pa 2150 2160}%
\special{pa 2150 1840}%
\special{fp}%
}}%
%
{\color[named]{Black}{%
\special{pn 8}%
\special{ar 2630 1840 480 480  3.1415927 4.7123890}%
}}%
%
{\color[named]{Black}{%
\special{pn 8}%
\special{pa 2630 1360}%
\special{pa 2950 1360}%
\special{fp}%
}}%
%
{\color[named]{Black}{%
\special{pn 8}%
\special{ar 2950 1840 480 480  4.7123890 6.2831853}%
}}%
%
{\color[named]{Black}{%
\special{pn 8}%
\special{pa 3430 1840}%
\special{pa 3430 2160}%
\special{fp}%
}}%
%
{\color[named]{Black}{%
\special{pn 8}%
\special{ar 2950 2160 480 480  6.2831853 6.2831853}%
\special{ar 2950 2160 480 480  0.0000000 1.5707963}%
}}%
%
{\color[named]{Black}{%
\special{pn 8}%
\special{pa 2950 2640}%
\special{pa 2470 2640}%
\special{fp}%
}}%
%
{\color[named]{Black}{%
\special{pn 8}%
\special{ar 2470 2160 480 480  1.5707963 3.1415927}%
}}%
%
{\color[named]{Black}{%
\special{pn 8}%
\special{pa 1990 2160}%
\special{pa 1990 1840}%
\special{fp}%
}}%
%
{\color[named]{Black}{%
\special{pn 8}%
\special{ar 2630 1840 640 640  3.1415927 4.7123890}%
}}%
%
{\color[named]{Black}{%
\special{pn 8}%
\special{pa 2630 1200}%
\special{pa 2950 1200}%
\special{fp}%
}}%
%
{\color[named]{Black}{%
\special{pn 8}%
\special{ar 2950 1840 640 640  4.7123890 6.2831853}%
}}%
%
{\color[named]{Black}{%
\special{pn 8}%
\special{pa 3590 1840}%
\special{pa 3590 2160}%
\special{fp}%
}}%
%
{\color[named]{Black}{%
\special{pn 8}%
\special{ar 2630 2160 960 640  6.2831853 6.2831853}%
\special{ar 2630 2160 960 640  0.0000000 1.5707963}%
}}%
%
{\color[named]{Black}{%
\special{pn 13}%
\special{pa 1990 2800}%
\special{pa 3590 2800}%
\special{fp}%
}}%
%
{\color[named]{Black}{%
\special{pn 4}%
\special{sh 1}%
\special{ar 2630 2800 26 26 0  6.28318530717959E+0000}%
}}%
%
{\color[named]{Black}{%
\special{pn 13}%
\special{ar 2790 2000 160 80  3.1415927 6.2831853}%
}}%
%
{\color[named]{Black}{%
\special{pn 13}%
\special{ar 2790 2000 200 72  6.2831853 6.2831853}%
\special{ar 2790 2000 200 72  0.0000000 3.1415927}%
}}%
%
{\color[named]{Black}{%
\special{pn 4}%
\special{sh 1}%
\special{ar 2470 2640 26 26 0  6.28318530717959E+0000}%
}}%
%
{\color[named]{Black}{%
\special{pn 4}%
\special{sh 1}%
\special{ar 2470 2480 26 26 0  6.28318530717959E+0000}%
}}%
%
{\color[named]{Black}{%
\special{pn 4}%
\special{sh 1}%
\special{ar 2470 2320 26 26 0  6.28318530717959E+0000}%
}}%
%
{\color[named]{Black}{%
\special{pn 8}%
\special{pa 2470 2000}%
\special{pa 2510 2080}%
\special{fp}%
\special{pa 2470 2000}%
\special{pa 2430 2080}%
\special{fp}%
}}%
%
{\color[named]{Black}{%
\special{pn 8}%
\special{pa 2310 2000}%
\special{pa 2350 2080}%
\special{fp}%
\special{pa 2310 2000}%
\special{pa 2270 2080}%
\special{fp}%
}}%
%
{\color[named]{Black}{%
\special{pn 8}%
\special{pa 2150 2000}%
\special{pa 2190 2080}%
\special{fp}%
\special{pa 2150 2000}%
\special{pa 2110 2080}%
\special{fp}%
}}%
%
{\color[named]{Black}{%
\special{pn 8}%
\special{pa 1990 2000}%
\special{pa 2030 2080}%
\special{fp}%
\special{pa 1990 2000}%
\special{pa 1950 2080}%
\special{fp}%
}}%
%
{\color[named]{Black}{%
\special{pn 8}%
\special{pa 3110 2000}%
\special{pa 3150 1920}%
\special{fp}%
\special{pa 3110 2000}%
\special{pa 3070 1920}%
\special{fp}%
}}%
%
{\color[named]{Black}{%
\special{pn 8}%
\special{pa 3270 2000}%
\special{pa 3310 1920}%
\special{fp}%
\special{pa 3270 2000}%
\special{pa 3230 1920}%
\special{fp}%
}}%
%
{\color[named]{Black}{%
\special{pn 8}%
\special{pa 3430 2000}%
\special{pa 3470 1920}%
\special{fp}%
\special{pa 3430 2000}%
\special{pa 3390 1920}%
\special{fp}%
}}%
%
{\color[named]{Black}{%
\special{pn 8}%
\special{pa 3590 2000}%
\special{pa 3630 1920}%
\special{fp}%
\special{pa 3590 2000}%
\special{pa 3550 1920}%
\special{fp}%
}}%
\put(26.3000,-29.6000){\makebox(0,0)[lb]{$*$}}%
\end{picture}%
\end{figure}
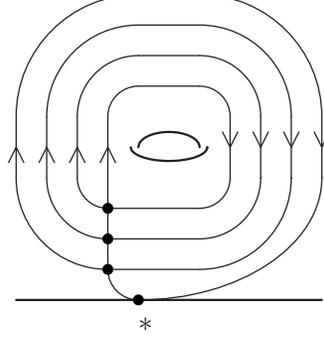

In \cite{KK1}, it was shown that the map $\mu$ extends to a map between completions $\mu\colon \widehat{\mathbb{Q}\pi_1(S)}\to \widehat{\mathbb{Q}\pi_1(S)} \widehat{\otimes} \widehat{\mathbb{Q}\hat{\pi}(S)}$.
Here $\widehat{\mathbb{Q}\pi_1(S)}$ and $\widehat{\mathbb{Q}\hat\pi(S)}$ are the completions of the group ring $\mathbb{Q}\pi_1(S)$ and the Goldman-Turaev Lie bialgebra $\mathbb{Q}\hat\pi(S)/\mathbb{Q}\mathbf{1}$, respectively, with respect to the augmentation ideal of $\mathbb{Q}\pi_1(S)$.
See \cite{KKp}.
Then we can consider $\log \gamma=\sum_{i=1}^{\infty} ((-1)^{i+1}/i) (\gamma-1)^i\in \widehat{\mathbb{Q}\pi_1(S)}$.

\begin{thm}
\label{thmulog}
Let $\gamma \in \pi$ be represented by a simple loop, and assume the pair $(\dot{\gamma}(0), \dot{\gamma}(1))$ is a positive basis of the tangent space $T_*S$.
Then we have 
$$\mu(\log \gamma) = \dfrac{1}{2}1\widehat{\otimes}\vert\log \gamma\vert
-\sum^\infty_{k=1}\dfrac{B_{2k}}{(2k)!}\sum^{2k-1}_{p=0}
 (-1)^p\binom{2k}{p}(\log \gamma)^p\widehat{\otimes} 
\vert (\log \gamma)^{2k-p}\vert.$$
\end{thm}
\begin{proof}
We identify the ring $\mathbb{Q}[[X, Y]]$ with the complete  tensor product $\mathbb{Q}[[Z]]\widehat{\otimes}\mathbb{Q}[[Z]]$ by the map $X \mapsto Z\widehat{\otimes} 1$ and $Y \mapsto 1\widehat{\otimes} Z$.
Then the computation (\ref{eq:mu(Z)}) implies 
\begin{align}
\hat\mu(\log z) 
=& -1\widehat{\otimes} 1 -\dfrac12(\log z)\widehat{\otimes} 1 + \dfrac121\widehat{\otimes}(\log z) \nonumber\\
&- \sum^\infty_{k=1}\dfrac{B_{2k}}{(2k)!}\sum^{2k}_{p=0}
(-1)^p\binom{2k}{p} (\log z)^p\widehat{\otimes} (\log z)^{2k-p}.
\label{mulog}
\end{align}
Since the curve $\gamma$ satisfies the formula (\ref{eq:mu(gamma^k)}) and we agree that $\vert 1\vert = 0$, the theorem follows from (\ref{mulog}). 
\end{proof}

As an application of Theorem \ref{thmulog}, the second-named author gives an explicit tensorial description of the Turaev cobracket on any genus $0$ compact surface with respect to the standard group-like expansion \cite{Kpre}.
It seems to suggest a certain connection between the operation $\mu$, or equivalently, the Turaev cobracket, and the Kashiwara-Vergne problem in the formulation by Alekseev-Torossian \cite{AT}.

\end{document}